\documentclass[11pt]{amsart}

\let\oldmarginpar\marginpar
\renewcommand\marginpar[1]{\-\oldmarginpar[\raggedleft\footnotesize #1]
{\raggedright\footnotesize #1}}

\usepackage{mathptmx}
\usepackage{lineno}
\usepackage{amsmath}
\usepackage{amscd}
\usepackage{amssymb}
\usepackage{amsthm}
\usepackage{xspace}
\usepackage[all,tips]{xy}
\usepackage[dvips]{graphicx}
\usepackage{verbatim}
\usepackage{syntonly}
\usepackage{hyperref}
\usepackage{color}
\usepackage{overpic}

\author{Larsen Louder}
\address{Department of Mathematics, University College London, London, England}
\email{l.louder@ucl.ac.uk}

\author{D.~B.~McReynolds}
\address{Department of Mathematics\\Purdue University\\West Lafayette, IN USA}
\email{dmcreyno@purdue.edu}

\author{Priyam Patel}
\address{Department of Mathematics\\ University of California, Santa Barbara\\ Santa Barbara, CA USA}
\email{patel@math.ucsb.edu}

\newtheorem{thm}{Theorem}[section]
\newtheorem{cor}[thm]{Corollary}
\newtheorem{ques}[thm]{Question}

\theoremstyle{definition}

\newtheorem{lemma}[thm]{Lemma}

\theoremstyle{remark}
\newtheorem{remark}{Remark}

\theoremstyle{definition}

\DeclareMathOperator{\Aut}{Aut}

\DeclareMathOperator{\GL}{GL}

\DeclareMathOperator{\D}{D}

\DeclareMathOperator{\Eval}{Eval}\DeclareMathOperator{\Hom}{Hom}

\newcommand{\vp}{\varphi}

\newcommand{\Sep}{\mathbf{Sep}}

\newcommand{\iny}{\infty}

\newcommand{\norm}[1]{\left\vert \left\vert #1\right\vert\right\vert}
\newcommand{\innp}[1]{\left< #1 \right>}
\newcommand{\abs}[1]{\left\vert#1\right\vert}

\newcommand{\set}[1]{\left\{#1\right\}}

\newcommand{\su}{\subset}

\newcommand{\lra}{\longrightarrow}

\newcommand{\B}[1]{\ensuremath{\mathbf{#1}}}

\newcommand{\N}{\ensuremath{\mathbf{N}}}
\newcommand{\Q}{\ensuremath{\mathbf{Q}}}

\newcommand{\Z}{\ensuremath{\mathbf{Z}}}
\newcommand{\C}{\ensuremath{\mathbf{C}}}

\usepackage{fancyhdr}

\fancypagestyle{plain}

\title{\textbf{Zariski Closures and Subgroup Separability}}
\usepackage{amsmath,amssymb,amsthm,url}

\begin{document}


\begin{abstract}
The main result of this article is a refinement of the well-known subgroup separability results of Hall and Scott for free and surface groups. We show that for any finitely generated subgroup, there is a finite dimensional representation of the free or surface group that separates the subgroup in the induced Zariski topology. As a corollary, we establish a polynomial upper bound on the size of the quotients used to separate a finitely generated subgroup in a free or surface group.
\end{abstract}

\maketitle

\section{Introduction}\label{S:Intro}

Given an algebraically closed field $\Omega$, a finite dimensional $\Omega$--vector space $V$, a finitely generated group $\Gamma$, and a homomorphism $\rho\colon \Gamma \to \GL(V)$, we have the subspace topology on $\rho(\Gamma)$ coming from the Zariski topology on $\GL(V) < \mathrm{End}(V)$. The pullback of this topology to $\Gamma$ under $\rho$ is called the Zariski topology associated to $\rho$. The primary goal of this article is to establish separability properties for $\Gamma$ by using Zariski topologies associated to finite dimensional representations. The foundational result was established by Mal'cev \cite{Mal} who proved that if $\rho\colon \Gamma \to \GL(V)$ is injective (e.g. $\Gamma$ is linear), then $\Gamma$ is residually finite. 

We say $\Gamma$ is \textbf{subgroup separable} (also called {\bf LERF}) if every finitely generated subgroup is closed in the profinite topology. Our main result shows that finitely generated subgroups of free and surface groups can be separated in the Zariski topology associated to a representation that depends on the subgroup.

\begin{thm}\label{T:MainZar}
Let $\Gamma$ be a free group with rank $r>1$ or the fundamental group of a closed surface $\Sigma_g$ with genus $g>1$. If $\Delta_0$ is a finitely generated subgroup of $\Gamma$, then there exists a faithful representation $\rho_{\Delta_0}\colon \Gamma \to \GL(V)$ such that $\overline{\rho_{\Delta_0}(\Delta_0)} \cap \rho_{\Delta_0}(\Gamma) = \rho_{\Delta_0}(\Delta_0)$, where $\overline{\rho_{\Delta_0}(\Delta_0)}$ is the Zariski closure of $\rho_{\Delta_0}(\Delta_0)$. That is, $\Delta_0$ is closed in the Zariski topology associated to $\rho_{\Delta_0}$.
\end{thm}

If a finitely generated subgroup $\Delta_0$ is Zariski closed in the sense above and $\gamma \notin \Delta_0$, then there is a homomorphism $\vp\colon \Gamma \to Q$ such that $\abs{Q}<\iny$ and $\vp(\gamma) \notin \vp(\Delta_0)$. Letting $\Lambda = \Delta_0 \cdot \ker(\vp) < \Gamma$, we see that $\Lambda$ is a finite index subgroup of $\Gamma$ of index at most $\abs{Q}$ with $\Delta_0 \leq \Lambda$ and $\gamma \notin \Lambda$. 

\begin{cor}\label{C:EffectiveSep}
Let $\Gamma$ be a free group with rank $r>1$ or the fundamental group of a closed surface $\Sigma_g$ with genus $g>1$, and let $\mathfrak{X}$ be a finite generating set for $\Gamma$ with $\norm{\cdot}_\mathfrak{X}$ the associated norm. If $\Delta_0 < \Gamma$ is a finitely generated subgroup, then there exists a constant $D > 0$ such that for each $\gamma \in \Gamma - \Delta_0$, there exists a homomorphism $\vp\colon \Gamma \to Q$ with $\vp(\gamma) \notin \vp(\Delta_0)$ and $\abs{Q} \leq \norm{\gamma}_\mathfrak{X}^D$. Letting $\Lambda = \Delta_0 \cdot \ker(\vp)$, $\Lambda$ is a finite index subgroup of $\Gamma$, of index at most $\abs{Q} \leq \norm{\gamma}_\mathfrak{X}^D$, such that $\Delta_0 \leq \Lambda$ and $\gamma \notin \Lambda$. Moreover, the index of the normal core of the subgroup $\Lambda$ is bounded above by $\abs{Q}$.
\end{cor}

Deducing Corollary \ref{C:EffectiveSep} from Theorem \ref{T:MainZar} is straightforward and uses methods from \cite{BouMcR15}. The constant $D$ explicitly depends on the subgroup $\Delta_0$ and the dimension of $V$ in Theorem \ref{T:MainZar}. For a general finite index subgroup, the crude upper bound for the index of the normal core is factorial in the index of the subgroup. It is for this reason that we include the statement regarding the normal core of $\Lambda$ at the end of Corollary \ref{C:EffectiveSep}. 

Recently, several effective separability results have been established; see \cite{BouRabee10}, \cite{BHP}, \cite{BK}, \cite{BM11}, \cite{BouMcR15} \cite{Buskin}, \cite{HP}, \cite{KM}, \cite{KT}, \cite{KMS}, \cite{LLM}, \cite{Patel}, \cite{Patel2}, \cite{Peng}, \cite{Rivin}, and \cite{Thom}. Most relevant here are the papers \cite{HP} and \cite{Patel} where bounds on the index of the separating subgroups for free and surface groups given. We compare the bounds of Corollary \ref{C:EffectiveSep} to the results in \cite{HP} and \cite{Patel} in \S \ref{S:FinalRemarks}.


\paragraph{\textbf{Acknowledgements.}} The authors would like to thank Khalid Bou-Rabee, Mark Hagen, Feng Luo, and Alan Reid for conversations on the work presented here. The second author was partially supported by NSF grants DMS-1105710 and DMS-1408458.

\section{Preliminaries}\label{S:Prelim}

\paragraph{\textbf{Complex algebraic groups.}}

Given a complex algebraic group $\B{G}<\GL(n,\C)$, there exist polynomials $Q_1,\dots,Q_r \in \C[X_{i,j}]$ such that
\[ \B{G} = \B{G}(\C) = V(Q_1,\dots,Q_r) = \set{X \in \C^{n^2}~:~Q_k(X) = 0,~k=1,\dots,r}. \]
We refer to the polynomials $Q_1,\dots,Q_r$ as defining polynomials for $\B{G}$. We will say that $\B{G}$ is $K$--defined for a subfield $K \su \C$ if there exists defining polynomials $Q_1,\dots,Q_r \in K[X_{i,j}]$ for $\B{G}$. For a complex affine algebraic subgroup $\B{H}<\B{G} < \GL(n,\C)$, we will pick the defining polynomials for $\B{H}$ to contain a defining set for $\B{G}$ as a subset. Specifically, we have polynomials $Q_1,\dots,Q_{r_\B{G}},Q_{r_\B{G}+1},\dots,Q_{r_\B{H}} \in \C[X_{i,j}]$ such that
\begin{equation}\label{E:CompatibleDefPolys}
\B{H} = V(Q_1,\dots,Q_{r_\B{H}}), \quad \B{G} = V(Q_1,\dots,Q_{r_\B{G}}).
\end{equation}
If $\B{G}$ is defined over a number field $K$ with associated ring of integers $\mathcal{O}_K$, we can find polynomials $Q_1,\dots,Q_r \in \mathcal{O}_K[X_{i,j}]$ as a defining set by clearing denominators. In the case when $K=\Q$ and $\mathcal{O}_K = \Z$, these are multivariable integer polynomials.

\paragraph{\textbf{Spaces of representations.}}

For a fixed finite set $\mathfrak{X} = \set{x_j}_{j=1}^t$ with associated free group $F(\mathfrak{X})$ and any group $G$, the set of homomorphisms from $F(\mathfrak{X})$ to $G$, denoted by $\Hom(F(\mathfrak{X}),G)$, can be identified with $G^t$. For any point $(g_1,\dots,g_t) \in G^t$, we have an associated homomorphism $\vp_{(g_1,\dots,g_t)}\colon F(\mathfrak{X}) \to G$ given by $\vp_{(g_1,\dots,g_t)}(x_i) = g_i$. For any word $w \in F(\mathfrak{X})$, we have a function $\Eval_w\colon \Hom(F(\mathfrak{X}),G) \to G$ defined by $\Eval_w(\vp_{(g_1,\dots,g_t)}) = \vp_{(g_1,\dots,g_t)}(w) = w(g_1,\dots,g_t)$. For a finitely presented group $\Gamma$, we fix a finite presentation $\innp{\mathfrak{X};\mathfrak{R}}$ where $\mathfrak{X} = \set{\gamma_1,\dots,\gamma_t}$ is a generating set (as a monoid) and $ \mathfrak{R} = \set{r_1,\dots,r_{t'}}$ is a finite set of relations. If $\B{G}$ is a complex affine algebraic subgroup of $\GL(n,\C)$, the set $\Hom(\Gamma,\B{G})$ of homomorphisms $\rho\colon \Gamma \to \B{G}$ can be identified with an affine algebraic subvarieity of $\B{G}^t$. Specifically
\begin{equation}\label{E:RepSpaceDef}
\Hom(\Gamma,\B{G}) = \set{(g_1,\dots,g_t) \in \B{G}^t~:~r_j(g_1,\dots,g_t) = I_n \text{ for all }j}.
\end{equation}
If $\Gamma$ is finitely generated, $\Hom(\Gamma,\B{G})$ is an affine algebraic variety by the Hilbert Basis Theorem. 

$\Hom(\Gamma,\B{G})$ also has a topology induced by the analytic topology on $\B{G}^t$. There is a Zariski open subset of $\Hom(\Gamma,\B{G})$ that is smooth in the this topology called the smooth locus, and the functions $\Eval_\gamma\colon \Hom(\Gamma,\B{G}) \to \B{G}$ are analytic on the smooth locus. For any subset $S \su \Gamma$ and representation $\rho \in \Hom(\Gamma,\B{G})$, $\overline{\rho(S)}$ will denote the Zariski closure of $\rho(S)$ in $\B{G}$.

\paragraph{\textbf{Effective separability functions.}}

For a finitely generated group $\Gamma$ with a fixed finite generating set $\mathfrak{X}$, we denote the associated norm by $\norm{\cdot}_\mathfrak{X}$. Given a subgroup $\Delta_0 < \Gamma$ and $\gamma \in \Gamma - \Delta_0$, we define
\[ \D_\Gamma(\Delta_0,\gamma) = \min\set{[\Gamma:\Lambda]~:~\Delta_0 < \Lambda,~\gamma \notin \Lambda}. \]
When $\Delta_0$ is separable in $\Gamma$, $\D(\Delta_0,\gamma) < \iny$ for all $\gamma \in \Gamma - \Delta_0$. The maximal value of $\D(\Delta_0,\gamma)$ ranging over all $\gamma \in \Gamma - \Delta_0$ with $\norm{\gamma}_\mathfrak{X} \leq m$ will be denoted by $\Sep_\Gamma(\Delta_0,m)$. Note that in \cite{HP}, $\D_\Gamma(\Delta_0, \gamma)$ and $\Sep_\Gamma(\Delta_0,m)$ are denoted by $D_\Gamma^{\Omega_{\Delta_0}}(\gamma)$ and $\Sep_{\Gamma, \mathfrak{X}}(\Delta_0,m)$, respectively. 

Recall that for a pair of functions $f_1,f_2\colon \N \to \N$, we say $f_1 \preceq f_2$ if there exists a constant $C > 0$ such that $f_1(m) \leq Cf_2(Cm)$ for all $m$. When $f_1 \preceq f_2$ and $f_2 \preceq f_1$, we write $f_1 \approx f_2$. The function $\Sep_\Gamma(\Delta_0,m)$ above depends on the choice of the generating set $\mathfrak{X}$. However, it is straightforward to see verify that $\Sep_{\Gamma,\mathfrak{X}}(\Delta_0,m) \approx \Sep_{\Gamma,\mathfrak{X}'}(\Delta_0,m)$ holds for any finite generating sets $\mathfrak{X},\mathfrak{X}'$ of $\Gamma$. We will suppress any dependence of the function $\Sep_\Gamma(\Delta_0,m)$ on the generating set $\mathfrak{X}$.

\section{Evaluation maps}\label{S:EvalMaps}

Throughout this section, $\Gamma$ will be a finitely generated group and $\Delta_0$ a finitely generated subgroup of $\Gamma$. For a complex affine algebraic group $\B{G}$ and any representation $\rho_0 \in \Hom(\Gamma,\B{G})$, we have the closed affine subvariety 
\[ \mathcal{R}_{\rho_0,\Delta_0}(\Gamma,\B{G}) = \set{\rho \in \Hom(\Gamma,\B{G})~:~\rho_0(\delta) = \rho(\delta) \text{ for all }\delta \in \Delta_0}. \]
We say that $\rho_0$ \textbf{distinguishes} $\Delta_0$ from $\gamma$ if the restriction of $\Eval_\gamma$ to $\mathcal{R}_{\rho_0,\Delta_0}(\Gamma,\Delta)$ is non-constant, that is to say, there exists $\rho \in \Hom(\Gamma, G)$ such that $\rho|_{\Delta_0} = \rho_0$ and $\rho(\gamma) \neq \rho_0(\gamma)$. We say that $\rho_0$ \textbf{weakly distinguishes} $\Delta_0$ in $\Gamma$, if $\rho_0$ distinguishes $\Delta_0$ from $\gamma$ for all $\gamma \in \Gamma - \Delta_0$. We say that $\rho_0$ \textbf{distinguishes} $\Delta_0$ in $\Gamma$ if for each finite set $S \su \Gamma - \Delta_0$, there are $\rho,\rho' \in \mathcal{R}_{\rho_0,\Delta_0}(\Gamma,\B{G})$ such that $\Eval_\gamma(\rho) \ne \Eval_\gamma(\rho')$ for all $\gamma \in S$. Finally, we say that $\rho_0$ \textbf{strongly distinguishes} $\Delta_0$ in $\Gamma$ if there are $\rho,\rho'\in\mathcal{R}_{\rho_0,\Delta_0}(\Gamma,\B{G})$ such that $\rho(\gamma) \ne \rho'(\gamma)$ for all $\gamma \in \Gamma - \Delta_0$.

\begin{lemma}\label{L:StrongToZar}
Let $\Gamma$ be a finitely generated group, $\B{G}$ is a complex algebraic group, and $\Delta_0$ a finitely generated subgroup of $\Gamma$. If $\Delta_0$ is strongly distinguished by a representation $\rho_0 \in \Hom(\Gamma,\B{G})$, then there exists a representation $\Phi\colon \Gamma \to \B{G} \times \B{G}$ such that $\Phi(\Gamma) \cap \overline{\Phi(\Delta_0)}  = \Phi(\Delta_0)$, where $\overline{\Phi(\Delta_0)}$ is the Zariski closure of $\Phi(\Delta_0)$ in $\B{G} \times \B{G}$.
\end{lemma}

\begin{proof}
By definition, there are representations $\rho,\rho' \in \mathcal{R}_{\rho_0,\Delta_0}(\Gamma,\B{G})$ such that $\gamma(\rho) \ne \gamma(\rho')$ for all $\gamma \in \Gamma - \Delta_0$. Take $\Phi\colon \Gamma \to \B{G} \times \B{G}$ given by $\Phi = \rho \times \rho'$. By construction, $\Phi(\Delta_0) < \mathrm{Diag}(\B{G})$ and $\Phi(\gamma) \notin \mathrm{Diag}(\B{G})$ for all $\gamma \in \Gamma - \Delta_0$. In particular, $\overline{\Phi(\Delta_0)} < \mathrm{Diag}(\B{G})$ since $\mathrm{Diag}(\B{G})$ is Zariski closed. Hence, $\Phi(\Delta_0) = \overline{\Phi(\Delta_0)} \cap \Phi(\Gamma)$.
\end{proof}

\begin{lemma}\label{L:WeakToStrong}
Let $\Gamma$ be a finitely generated group, $\B{G}$ a complex algebraic group, and $\Delta_0$ a finitely generated subgroup of $\Gamma$. If $\Delta_0$ is distinguished by a representation $\rho_0 \in \Hom(\Gamma,\B{G})$, then $\rho_0$ strongly distinguishes $\Delta_0$.
\end{lemma} 

\begin{proof}
We order $\Gamma - \Delta_0 = \set{\gamma_1,\gamma_2,\dots}$ and for each $j \in \N$, define $S_j = \set{\gamma_i}_{i=1}^j$. As $\rho_0$ distinguishes $\Delta_0$, for each $j \in \N$, there exists $\rho_j \in \Hom(\Gamma,\B{G})$ such that $\rho_j(\delta) = \rho_0(\delta)$ for all $\delta \in \Delta_0$ and $\rho_j(\gamma_i) \ne \rho_0(\gamma_i)$ for all $1 \leq i \leq j$. Selecting a non-principal ultrafilter $\omega$ on $\N$, we have the associated ultraproduct representation $\rho_\omega\colon \Gamma \to \B{G}$ (cf  \cite{Hal95}). If $\gamma \in \Gamma - \Delta_0$, then $\rho_j(\gamma) \ne \rho_0(\gamma)$ for a cofinite set of $j \in \N$ and so $\rho_\omega(\gamma) \ne \rho_0(\gamma)$. Similar, if $\delta \in \Delta_0$, then $\rho_j(\delta) = \rho_0(\delta)$ for all $j \in \N$ and so $\rho_\omega(\delta) = \rho_0(\delta)$.  In particular, $\rho_0$ strongly distinguishes $\Delta_0$.
\end{proof}

\begin{remark}
Lemma \ref{L:WeakToStrong} can also be proved using the Baire Category Theorem.
\end{remark}

\begin{cor}\label{P:MainTechProp}
Let $\Gamma$ be a finitely generated group, $\B{G}$ is a complex algebraic group, and $\Delta_0$ a finitely generated subgroup of $\Gamma$. If $\Delta_0$ is distinguished by a representation $\rho_0 \in \Hom(\Gamma,\B{G})$, then there exists a representation $\Phi\colon \Gamma \to \B{G} \times \B{G}$ such that $\Phi(\Gamma) \cap \overline{\Phi(\Delta_0)}  = \Phi(\Delta_0)$, where $\overline{\Phi(\Delta_0)}$ is the Zariski closure of $\Phi(\Delta_0)$ in $\B{G} \times \B{G}$.
\end{cor}

\begin{proof}
Since $\Delta_0$ is distinguished by $\rho_0$, it follows from Lemma \ref{L:WeakToStrong} that $\Delta_0$ is strongly distinguished by $\rho_0$. Hence, by Lemma \ref{L:StrongToZar}, we obatin the desired representation $\Phi\colon \Gamma \to \B{G} \times \B{G}$.
\end{proof}

\subsection{Twisting by automorphisms}

Given an automorphism $\psi_0 \in \Aut(\Gamma)$, we define
\[ \Aut_{\psi_0,\Delta_0}(\Gamma) = \set{\psi \in \Aut(\Gamma)~:~\psi_{|\Delta_0} = (\psi_0)_{|\Delta_0}}. \]
For each $\gamma \in \Gamma$, we have the function $\Eval_{\Aut,\gamma}\colon \Aut(\Gamma) \to \Gamma$ defined by $\Eval_{\Aut,\gamma}(\psi) = \psi(\gamma)$. We say that $\Delta_0$ is \textbf{weakly $\psi_0$--distinguished in $\Gamma$} if $\Eval_{\Aut,\gamma}$ is non-constant on $\Aut_{\psi_0,\Delta_0}(\Gamma)$ for all $\gamma \in \Gamma - \Delta_0$. We say that $\Delta_0$ is \textbf{$\psi_0$--distinguished} if for any finite set $S$ of $\Gamma - \Delta_0$, there are automorphisms $\psi_S,\psi_S' \in \Aut_{\psi_0,\Delta_0}(\Gamma)$ such that $\Eval_{\Aut,\gamma}(\psi_S) \ne \Eval_{\Aut,\gamma}(\psi_S')$, i.e. $\psi_S(\gamma) \ne \psi'_S(\gamma)$, for all $\gamma \in S$. Finally, we say $\Delta_0$ is \textbf{strongly $\psi_0$--distinguished} if there exist $\psi,\psi'\in \Aut_{\psi_0,\Delta_0}(\Gamma)$ such that $\psi(\gamma) \ne \psi'(\gamma)$ for all $\gamma \in \Gamma - \Delta_0$.

\begin{lemma}\label{L:Auto-to-Rep}
If $\Gamma$ is a finitely generated group and $\Delta_0$ is (weakly, strongly) $\psi_0$--distinguished in $\Gamma$, then for any complex algebraic group $\B{G}$ and any injective representation $\rho \in \mathcal{R}(\Gamma,\B{G})$, $\Delta_0$ is (weakly, strongly) distinguished by $\rho \circ \psi_0$ in $\Gamma$.
\end{lemma}

\begin{proof}
For any $\psi,\psi' \in \Aut_{\psi_0,\Delta_0}(\Gamma)$ and $\rho \in \Hom(\Gamma,\B{G})$, we have
$(\rho \circ \psi)_{|\Delta_0} = (\rho \circ \psi')_{|\Delta_0}$. In particular, for each $\gamma \in \Gamma - \Delta_0$, there exists $\psi,\psi' \in \Aut_{\psi_0,\Delta_0}(\Gamma)$ such that $\Eval_{\Aut,\gamma}(\psi) \ne \Eval_{\Aut,\gamma}(\psi')$ since $\Delta_0$ is weakly $\psi_0$--distinguished. As $\rho$ is injective, $\rho(\psi(\gamma)) \ne \rho(\psi'(\gamma))$ and so $\Eval_\gamma(\rho \circ \psi) \ne \Eval_\gamma(\rho \circ \psi')$. By definition, $\rho \circ \psi,\rho\circ\psi' \in \mathcal{R}_{\rho \circ \psi_0,\Delta_0}(\Gamma,\B{G})$ and so $\Delta_0$ is weakly distinguished by $\rho \circ \psi_0$. The proof when $\Delta_0$ is $\psi_0$--distinguished or strongly $\psi_0$--distinguished is identical.
\end{proof}

\section{Proof of Theorem \ref{T:MainZar}}\label{S:MainProof}

Before proving Theorem \ref{T:MainZar}, we require a pair of lemmas.

\begin{lemma}\label{L:FreeAutoTwist}
If $\Lambda = \Delta_0 \ast \Delta$ with $\Delta_0 \neq \{1\}$, then there exists an automorphism of $\Lambda$ whose set of fixed points is exactly $\Delta_0$. In particular, $\Delta_0$ is strongly $\psi_0$--distinguished, where $\psi_0$ is the identity automorphism. 
\end{lemma}

\begin{proof}
We assume $\Lambda \ne \Delta_0$ as that case is trivial. Fix $\delta$ a nontrivial element in $\Delta_0$. Define an automorphism $\psi\colon \Lambda \rightarrow \Lambda$ as being the identity on $\Delta_0$ and $\psi(k) = \delta\cdot k \cdot \delta^{-1}$ for all $k \in \Delta$. Given $\gamma \in \Lambda -\Delta_0$, we have a reduced expression $\gamma = h_1 k_1 \cdots h_m k_m$, where $m \geq 1$, $h_i \in  \Delta_0 -\{1\}$, and $k_i \in \Delta - \{1\}$, with the exception that $h_1$ or $k_m$ could be trivial. Thus,
\[ \psi(\gamma) = \psi(h_1 k_1 \cdots h_m k_m)= (h_1\delta)k_1(\delta^{-1}h_2\delta )\cdots  (\delta^{-1}h_m\delta)k_m(\delta^{-1}) = h_1'k_1h_2'\cdots h_m'k_m\delta^{-1}, \] 
where $h_i \neq h_i'$ for all $i=1, \dots, m$. Thus, $\psi(\gamma) \neq \gamma$ for all $\gamma \in \Lambda - \Delta_0$, so that the set of fixed points of $\psi$ is exactly $\Delta_0$. 
\end{proof}

\begin{lemma}\label{L:SurfaceAutoTwist}
If $\Sigma_g$ is a closed surface of genus $g>1$ and $\Sigma'$ is a compact, embedded, incompressible subsurface, then $\pi_1(\Sigma', p)$ is strongly $\psi_0$--distinguished in $\pi_1(\Sigma_{g}, p)$, where $\psi_0$ is the identity.
\end{lemma}

\begin{proof}
We assume $\pi_1(\Sigma_g) \ne \pi_1(\Sigma')$ as the alternative is trivial. We need $\psi \in \Aut_{\psi_0}(\pi_1(\Sigma_g),\pi_1(\Sigma'))$ with $\psi([\gamma]) \neq \psi_0([\gamma]) = [\gamma]$ for all $[\gamma] \in \pi_1(\Sigma_g) - \pi_1(\Sigma')$. Fixing $p \in \textrm{Int}(\Sigma')$, for every $[\gamma]$ in $\pi_1(\Sigma_g, p) - \pi_1(\Sigma', p)$ and any loop $c$ representing $[\gamma]$, we must have $c \cap \partial \Sigma' \ne \emptyset$. For each boundary component $\alpha_i$, set $\tau_i\colon \Sigma_g \rightarrow \Sigma_g$ to be Dehn twist about $\alpha_i$ for $i = 1, \dots, b$, and note  that $\tau_i$ induces an automorphism $\psi_i \in \Aut_{\psi_0}(\pi_1(\Sigma_g, p),\pi_1(\Sigma',p))$ defined by $\psi_i([\gamma]) = [\tau_i(\gamma)]$. Thus, for any $[\gamma] \in \pi_1(\Sigma_g) - \pi_1(\Sigma')$, $\psi_i([\gamma]) \neq [\gamma]$ for some $i$, and setting $\psi = \psi_b \circ \cdots \circ \psi_1$ completes the proof. 
\end{proof}

\begin{proof}[Proof of Theorem \ref{T:MainZar}]
Let $\Gamma$ be either a free group of rank $r>1$ or the fundamental group of a closed surface $\Sigma_g$ of genus $g>1$. Given a finitely generated subgroup $\Delta_0$, if $\Gamma$ is free, then by Hall \cite{Hall}, there exists a finite index subgroup $\Lambda < \Gamma$ with $\Lambda = \Delta_0 \ast \Delta$. If $\Gamma$ is the fundamental group of a closed surface, then by Scott \cite{Scott}, there is a finite cover $P\colon \Sigma_{g_0} \to \Sigma_g$ such that $\Delta_0 < P_*(\pi_1(\Sigma_{g_0}))=\Lambda$. Moreover, there exists a embedded compact subsurface $\Sigma_{\Delta_0}$ of $\Sigma_{g_0}$ with $\Delta_0 = \pi_1(\Sigma_{\Delta_0})$. In either case, we can apply Lemma \ref{L:FreeAutoTwist} or Lemma \ref{L:SurfaceAutoTwist} to see that $\Delta_0$ is strongly $\psi_0$--distinguished in $\Lambda$. For any faithful representation $\rho_0\in \Hom(\Lambda,\GL(2,\C))$, we see that $\rho_0$ strongly distinguishes $\Delta_0$ by Lemma \ref{L:Auto-to-Rep}. By Corollary \ref{P:MainTechProp}, we have $\Phi\colon \Lambda \to \GL(2,\C) \times \GL(2,\C)$ such that $\Phi(\gamma) \in \mathrm{Diag}(\GL(2,\C))$ if and only if $\gamma \in \Delta_0$. Setting $d_{\Delta_0} = [\Gamma:\Lambda]$, we have the induced representation $\mathrm{Ind}_{\Lambda}^{\Gamma}(\Phi)\colon \Gamma \to \GL(2d_{\Delta_0},\C) \times \GL(2d_{\Delta_0},\C)$. Taking $\rho = \mathrm{Ind}_{\Lambda}^\Gamma(\Phi)$, it follows by the construction of $\rho$ and from the definition of induction that, $\rho(\gamma) \in \overline{\rho(\Delta_0)}$ if and only if $\gamma \in \Delta_0$. In particular, $\rho(\Delta_0) = \rho(\Gamma) \cap \overline{\rho(\Delta_0)}$ as needed, and $\rho$ is faithful since $\rho_0$ is faithful.
\end{proof}

\section{Proof of Corollary \ref{C:EffectiveSep}}\label{S:Effective}

 The following basic result has been proven in \cite{Berg}, \cite{MS}, and \cite{McR04}.

\begin{lemma}\label{L:BasicLemma1}
Let $\B{G} < \GL(n,\C)$ be a $\overline{\Q}$--algebraic group, $\B{H} < \B{G}$ a $\overline{\Q}$--algebraic subgroup, $\Gamma < \B{G}$ a finitely generated subgroup. If $\Delta_0 = \B{H} \cap \Gamma$, then $\Delta_0$ is closed in the profinite topology.
\end{lemma}

We include a proof here as it is required in the proof of Corollary \ref{C:EffectiveSep}.

\begin{proof}
Given $\gamma \in \Gamma - \Delta_0$, we require a homomorphism $\vp\colon \Gamma \to Q$ with $\abs{Q}<\iny$ and $\vp(\gamma) \notin \vp(\Delta_0)$. We first select polynomials $Q_1,\dots,Q_{r_\B{G}},\dots,Q_{r_\B{H}} \in \C[X_{i,j}]$ satisfying \eqref{E:CompatibleDefPolys}. Since $\B{G},\B{H}$ are $\overline{\Q}$--defined, we can select $Q_j \in \mathcal{O}_{K_0}[X_{i,j}]$ for some number field $K_0/\Q$. We fix a finite set $\set{\gamma_1,\dots,\gamma_{r_\Gamma}}$ that generates $\Gamma$ as a monoid. In order to distinguish between the elements of $\Gamma$ as an abstract group versus the explicit elements in $\B{G}$, we set $\gamma = A_{\gamma} \in \B{G}$ for each $\gamma \in \Gamma$. In particular, we have a representation $\rho_0\colon \Gamma \lra \B{G}$
given by $\rho_0(\gamma_t) = A_{\gamma_t}$. We set $K_\Gamma$ to be the field generated over $K_0$ by the set of matrix entries $\set{(A_{\gamma_t})_{i,j}}_{t,i,j}$. It is straightforward to see that $K_\Gamma$ is independent of the choice of the generating set for $\Gamma$. Since $\Gamma$ is finitely generated, the field $K_\Gamma$ has finite transcendence degree over $\Q$ and so $K_\Gamma$ is isomorphic to a field of the form $K(T)$ where $K/\Q$ is a number field and $T = \set{T_1,\dots,T_d}$ is a transcendental basis (see \cite[Cor.~3.3.3]{Roman}). For each $A_{\gamma_t}$, we have $(A_{\gamma_t})_{i,j} = F_{i,j,t}(T)\in K_\Gamma$. In particular, we can view the $(i,j)$--entry of the matrix $A_{\gamma_t}$ as a rational function in $d$ variables with coefficients in some number field $K$. Taking the ring $R_\Gamma$ generated over $\mathcal{O}_{K_0}$ by the set $\set{(A_{\gamma_t})_{i,j}}_{t,i,j}$, $R_\Gamma$ is obtained from $\mathcal{O}_K[T_1,\dots,T_d]$ by inverting a finite number of integers and polynomials. Any ring homomorphism $R_\Gamma \to R$ induces a group homomorphism $\GL(n,R_\Gamma) \to \GL(n,R)$, and as $\Gamma < \GL(n,R_\Gamma)$, we obtain $\Gamma \to \GL(n,R)$. If $\gamma \in \Gamma - \Delta_0$, then there exists $r_\B{G} < j_\gamma \leq r_\B{H}$ such that
$P_\gamma = Q_{j_\gamma}((A_\gamma)_{1,1},\dots,(A_\gamma)_{n,n}) \ne 0$. Using Lemma 2.1 in \cite{BouMcR15}, we have a ring homomorphism $\psi_R\colon R_\Gamma \to R$ with $\abs{R} < \iny$ such that $\psi_R(P_\gamma) \ne 0$. Setting $\rho_R\colon \GL(n,R_\Gamma) \to \GL(n,R)$, we assert that $\rho_R(\gamma) \notin \rho_R(\Delta_0)$. To see this, set $\overline{A}_{\eta} = \rho_R(\eta)$ for each $\eta \in \Gamma$, and note that $\psi_R(Q_j((A_\eta)_{1,1},\dots,(A_\eta)_{n,n})) = Q_j((\overline{A}_\eta)_{1,1},\dots,(\overline{A}_\eta)_{n,n})$. For each $\delta \in \Delta_0$, we know that $Q_{j_\gamma}((A_\delta)_{i,j}) = 0$ and so $Q_j((\overline{A}_\eta)_{1,1},\dots,(\overline{A}_\eta)_{n,n}) = 0$. However, by selection of $\psi_R$, we know that $\psi_R(P_\gamma) \ne 0$ and so $\rho_R(\gamma)\notin \rho_R(\Delta_0)$.
\end{proof}

\subsection{Proof of Corollary \ref{C:EffectiveSep}}

To prove Corollary \ref{C:EffectiveSep}, we combine Theorem \ref{T:MainZar} with Lemma \ref{L:BasicLemma1}. By Theorem \ref{T:MainZar}, there exists a representation $\rho\colon \Gamma \to \GL(n,\C)$ such that if $\B{G} = \overline{\rho(\Gamma)}$ and $\B{H} = \overline{\rho(\Delta_0)}$, then $\rho(\Delta_0) = \rho(\Gamma) \cap \B{H}$. We can construct $\rho_0$ in the proof of Theorem \ref{T:MainZar} so that $\B{G},\B{H}$ are both $\overline{\Q}$--defined. Consequently, we can use Lemma \ref{L:BasicLemma1} to separate $\Delta_0$ in $\Gamma$. In order to make Lemma \ref{L:BasicLemma1} effective we need to bound the order of the ring $R$ in terms of the word length of the element $\gamma$ in the proof of Lemma \ref{L:BasicLemma1}. Lemma 2.1 from \cite{BouMcR15} bounds the size of $R$ in terms of the coefficient size and degree of the polynomial $P_\gamma$. It follows from the discussion on p.~412--413 of \cite{BouMcR15} that the coefficients and degree can be bounded in terms of the word length of $\gamma$, and the coefficients and degrees of the polynomials $Q_j$. As the functions $Q_j$ are independent of the word $\gamma$, we see that there exists a constant $D_0$ such that $\abs{R} \leq \norm{\gamma}^{D_0}$. By construction, the group $Q$ needed in Corollary \ref{C:EffectiveSep} for $\gamma$ is a subgroup of $\GL(n,R)$ and so $\abs{Q} \leq \abs{R}^{n^2} \leq \norm{\gamma}^{D_0n^2}$. Hence, we can take $D = D_0n^2$.

\section{Final Remarks}\label{S:FinalRemarks}

The main contribution of Corollary \ref{C:EffectiveSep} is that we establish polynomial bounds on the size of the normal core of the finite index subgroup $\Lambda$ used in separating $\gamma$ from $\Delta_0$. The methods used in \cite{HP} give linear bounds in terms of the word length of $\gamma$ on the index of the subgroup used in the separation but do not easily produce polynomial bounds for the normal core of that finite index subgroup. With care taken to make our argument optimal, we can obtain bounds on the index of the separating subgroup on the order of magnitude $C\norm{\gamma}$ as well.

Finally, to what extent Theorem \ref{T:MainZar} can be generalized to other classes of groups is unclear. We make specific use of our settings but believe that the broad framework we present should work for a larger class of groups. That prompts the following question:

\begin{ques}
Does Theorem \ref{T:MainZar} hold when $\Gamma$ is the fundamental group of a closed hyperbolic 3--manifold and $\Delta_0$ is a finitely generated, geometrically finite subgroup? Does Theorem \ref{T:MainZar} hold when $\Gamma$ is a right-angled Artin group and $\Delta_0$ is a quasi-convex subgroup?
\end{ques}

By \cite{HP}, separability of these subgroups can be done with finite index subgroups of polynomial index in $\norm{\gamma}$. That is a necessary for Theorem \ref{T:MainZar} to hold. In the above question, we optimistically believe that this condition is sufficient.



\end{document}